\documentclass{amsart}
\usepackage{amssymb, amsmath, latexsym, mathabx, dsfont,tikz}
\usepackage{multirow}

\renewcommand{\baselinestretch}{\baselinestretch}
\renewcommand{\baselinestretch}{1.1}
\numberwithin{equation}{section}
\makeatletter
\newcommand{\skipitems}[1]{%
  \addtocounter{\@enumctr}{#1}%
}
\makeatother

\newtheorem{thm}{Theorem}[section]
\newtheorem{lem}[thm]{Lemma}

\newtheorem{prop}[thm]{Proposition}

\theoremstyle{definition}

\theoremstyle{remark}
\newtheorem{rmk}[thm]{Remark}

\numberwithin{equation}{section}
\newcommand{\lra}{{\longrightarrow}}

\newcommand{\gen}{\text{gen}}

\newcommand{\ord}{\text{ord}}

\newcommand{\z}{{\mathbb Z}}

\begin{document}
\title{A sum of squares not divisible by a prime}

\author{Kyoungmin Kim and Byeong-Kweon Oh}

\address{Department of Mathematics, Sungkyunkwan University, Suwon 16419, korea}
\email{kiny30@skku.edu}
\thanks{This work of the first author was supported by the National Research Foundation of Korea(NRF) grant funded by the Korea government(MSIT) (NRF-2016R1A5A1008055  and NRF-2018R1C1B6007778)}

\address{Department of Mathematical Sciences and Research Institute of Mathematics, Seoul National University, Seoul 08826, Korea}
\email{bkoh@snu.ac.kr}
\thanks{This work of the second author was supported by the National Research Foundation of Korea (NRF-2017R1A2B4003758).}



\begin{abstract}  Let $p$ be a prime. We define $S(p)$ the smallest number $k$ such that every positive integer is a sum of at most $k$ squares of integers that are not divisible by $p$. In this article, we prove that $S(2)=10$, $S(3)=6$, $S(5)=5$, and $S(p)=4$ for any prime $p$ greater than $5$.   In particular, it is proved that every positive integer  is a sum of at most four squares not divisible by $5$, except the unique positive integer $79$.
\end{abstract}

\maketitle


\section{Introduction}
The famous four square theorem says that every non-negative integer is a sum of at most $4$ squares, that is, for the quaternary quadratic form $f(x,y,z,t)=x^2+y^2+z^2+t^2$, the Diophantine equation $f(x,y,z,t)=n$ always has an integer solution for any non-negative integer $n$.  
  After Lagrange \cite {lag}  proved this celebrated theorem, it was generalized in several directions. Ramanujan \cite {ramanujan} determined that there are exactly $55$  positive definite integral diagonal quaternary quadratic forms. Later, Dickson  \cite {dic} confirmed Ramanujan's assertion is correct except the quaternary quadratic form $x^2+2y^2+5z^2+5t^2$, which represents all non-negative integers, except the unique integer $15$. Conway and Schneeberger  proved, so called, $15$-Theorem which says that any positive definite integral quadratic form representing  $1,2,3,5,6,7,10,14$, and $15$ represents all non-negative integers. Recently, Bhargava \cite {b} provided a very simple and elegant proof of $15$-Theorem.
 
  Another generalization was initiated by Mordell \cite {m} and Ko \cite {ko}. In those papers, they proved that every positive definite integral quadratic form of rank $n$ less than or equal to  $5$ is represented by a sum of $n+3$ squares. In fact, there is a quadratic form of rank $6$ that is not represented by a sum of any number of integral squares. One of such quadratic forms is the root lattice $E_6$.    
  
In this article, we generalize Lagrange's four square theorem in another direction. Let $p$ be a prime. We say an integer $n$ is  {\it a sum of $k$ squares not divisible by $p$} if there are integers $x_1,x_2,\dots,x_k$ such that
$$
n=x_1^2 +x_2^2+\cdots+x_k^2 \quad \text{and} \quad (p,x_1x_2\cdots x_k)=1.
$$        
We define $S(p)$ the smallest integer $k$ such that any positive integer is a sum of at most $k$ squares not divisible by $p$.  In this article, we prove that 
$$
S(2)=10,  \  \   \ S(3)=6,  \  \   \   S(5)=5, \ \  \  \text{and}  \  \  \  S(p)=4 \  \ \    \text{for any prime $p \ge 7$}.
$$ 
 In particular, it is proved that every positive integer is a sum of at most four squares not divisible by $5$,  except the unique integer $79$.
 
Throughout this article, we always assume that a quadratic form of rank $n$ 
$$
f(x_1,x_2,\dots,x_n)=\sum_{i,j=1}^n a_{ij} x_ix_j \qquad (a_{ij}=a_{ji})
$$
is positive definite and integral, that is $a_{ij} \in \z$ for any $i,j$. The corresponding symmetric matrix $M_f$ to the quadratic form $f$ is defined by $M_f=(a_{ij})$. The discriminant $df$ of the quadratic form $f$ is defined by the determinant of the corresponding symmetric matrix $M_f$. 
 If $f$ is diagonal, that is, $a_{ij}=0$ for any $i\ne j$, then we write 
$$
f = \langle a_{11},a_{22},\dots,a_{nn}\rangle.
$$
  We say an integer $a$ is {\it represented by $f$} if there are integers $x_1,x_2,\dots,x_n$ such that $a=f(x_1,x_2,\dots,x_n)$. In this case, we write $a\, \lra f$. 
  In particular, we say $a$ is {\it a sum of $k$ squares} if $a$ is represented by the quadratic form 
$I_k=\langle 1,1,\dots,1\rangle$. We define
$$
R(a,f)=\{ (x_1,x_2,\dots,x_n) \in \z^n : a=f(x_1,x_2,\dots,x_n)\} \  \  \text{and} \ \   r(a,f)=\vert R(a,f)\vert.
$$
Note that $r(a,f)$ is finite, for we are assuming that $f$ is positive definite. 

For two quadratic forms $f$ and $g$ of rank $n$, we say $f$ is {\it isometric} to $g$ if there is an integral matrix $T \in M_n(\z)$ such that $T^tM_fT=M_g$. We say $f$ is isometric to $g$ over the $p$-adic integer ring $\z_p$ if there is a matrix $T \in M_n(\z_p)$ satisfying the above property. The isometry group $O(f)$ of $f$ is defined by
$$
O(f)=\{ T \in M_n(\z) : T^tM_fT=M_f\} \ \  \text{and} \  \  o(f)=\vert O(f)\vert.
$$ 

 The genus $\gen(f)$ of $f$ is the set of all quadratic forms that are isometric to $f$ over  $\z_p$ for any prime $p$.  The class number $h(f)$ of $f$ is the number of  isometric classes in the genus of $f$.  We say an integer $a$ is {\it represented by the genus of $f$} if there is a quadratic form $f' \in \gen(f)$ that represents $a$.  Note that $a$ is represented by the genus of $f$ if and only if the equation $a=f(x_1,x_2,\dots,x_n)$ always has a solution $(x_1,x_2,\dots,x_n) \in \z_p^n$ for any prime $p$ (see, for example, 102:5 of \cite{om}). 

For a quadratic form $f$ and an integer $a$, we define
$$
w(f)=\sum_{[g] \in \gen(f)} \frac1{o(g)} \qquad \text{and} \qquad r(a,\gen(f))=\frac1{w(f)}\sum_{[g] \in \gen(f)} \frac{r(a,g)}{o(g)},
$$ 
where $[g]$ is the isometric class containing $g$ in the genus  of $f$. Note that if $h(f)=1$, then we have $r(a,\gen(f))=r(a,f)$. 

Any unexplained notations and terminologies can be found in  \cite {ki} or \cite {om}.

\section{A sum of squares not divisible by $2$ or $3$}

Let $p$ be a prime. We say that an integer $n$ is  {\it a sum of $k$ squares not divisible by $p$} if there are integers $x_1,x_2,\dots,x_k$ such that 
$$
n=x_1^2+x_2^2+\cdots+x_k^2 \quad \text{and} \quad (p,x_1x_2\cdots x_k)=1.
$$ 
If $n$  is a sum of $k$ squares not divisible by $p$, then we write  $n  \stackrel{p}{\lra} I_k$. We further define $S(p)$
the smallest integer $k$ such that any positive integer is a sum of at most $k$ squares not divisible by $p$. Note that $S(p)\ge 4$ for any prime $p$.

\begin{lem}\label {t-square}
Let $f$ be a ternary quadratic form and let $p$ be a prime not dividing $2df$. Let $n$ be a positive integer and let $\ord_p(n)=\lambda_p$. If $n$ is represented by the genus of $f$, then we have
$$ 
\begin{array}  {ll}
\displaystyle \frac{r(p^2n,\text{gen}(f))}{r(n,\gen(f))}
&=p\cdot\displaystyle  \!\prod_{q\vert 2df} \frac{\alpha_{q}(p^2n,f)}{\alpha_{q}(n,f)}\displaystyle  \prod_{q\nmid 2df} \frac{\alpha_{q}(p^2n,f)}{\alpha_{q}(n,f)}\\[0.5cm]
&=\displaystyle \left( \frac{ \displaystyle p^{\left[\frac{\lambda_p}{2}\right] +2}-1-\left(\frac{-np^{-2\left[\frac{\lambda_p}{2}\right]}\cdot df}{p}\right)\left(p^{\left[\frac{\lambda_p}{2}\right] +1}-1\right)}{ \displaystyle p^{\left[\frac{\lambda_p}{2}\right] +1}-1-\left(\frac{-np^{-2\left[\frac{\lambda_p}{2}\right]}\cdot df}{p}\right)\left(p^{\left[\frac{\lambda_p}{2}\right]}-1\right)} \right).
\end{array}
$$
Here $[x]$ is the greatest integer not exceeding $x$ and $\left(\frac{\cdot}{p}\right)$ is the Legendre symbol.
\end{lem}
\begin{proof}
By the Minkowski-Siegel formula, we have
$$
r(n,\text{gen}(f))=\pi^{\frac{3}{2}}\cdot\Gamma \left(\frac{3}{2}\right)^{-1}\cdot{\sqrt{\frac {n}{df}}}\cdot\prod_{q<\infty}\alpha_q({n,f}),
$$
where $\alpha_q$ is the local density over $\z_q$. Note that by Theorem 3.1 in \cite {y}, we have
$$
\alpha_p(n,f)=
\begin{cases}
\displaystyle 1+\frac{1}{p}-\frac{1}{p^{\frac{\lambda_p+1}2}}-\frac{1}{p^{\frac{\lambda_p+3}2}} & \text{if $\lambda_p$ is odd},\\[0.3cm]
\displaystyle 1+\frac{1}{p}-\frac{1}{p^{\frac{\lambda_p+2}{2}}}+\left(\frac{-p^{-\lambda_p}n\cdot df}{p}\right)\frac{1}{p^{\frac{\lambda_p+2}{2}}} & \text{otherwise}.
\end{cases}
$$
Hence the lemma follows directly from this.
\end{proof}

\begin{lem}\label {class1}
Let $f$ be a ternary quadratic form and let $n$ be a positive integer. If the class number of $f$ is 1, then for any prime $p$ not dividing  $2df$, we have
$$
r(p^2n,f)-r(n,f)>0,
$$
provided that $p^2n$ is represented by $f$.
\end{lem}     
\begin{proof} Since we are assuming that $h(f)=1$, we have by Lemma \ref {t-square},
$$
\frac {r(p^2n,f)}{r(n,f)}=\frac {r(p^2n,\gen(f))}{r(n,\gen(f))}>1.
$$
This completes the proof.   \end{proof}

\begin{prop} Every positive integer is a sum of at most $10$ squares of odd integers, and in fact, $S(2)=10$.
\end{prop}
\begin{proof}
 If $n\equiv 3 \pmod 8$, then by Legendre's three-square theorem, there are integers $a_1,a_2$, and $a_3$ such that
$$
n=a_1^2+a_2^2+a_3^2  \quad \text{and} \quad (2,a_1a_2a_3)=1.
$$
Hence $n \stackrel{2}{\lra} I_3$.  Assume the $n\equiv t \pmod 8$ for $3\le t \le 8$. Since $n-(t-3)\equiv 3 \pmod 8$, we have
$n \stackrel{2}{\lra} I_t$. Next assume that $n\equiv 1 \pmod 8$. If $n$ is a square of an integer, then $n \stackrel{2}{\lra} I_1$. If $n$ is not a square, then we have $n \stackrel{2}{\lra} I_9$, for $n-6\equiv 3 \pmod 8$. Finally, assume that $n \equiv 2 \pmod 8$. If $n$ is a sum of two squares, then $n \stackrel{2}{\lra} I_2$. If $n$ is not a sum of two squares, then $n \stackrel{2}{\lra} I_{10}$.  Note that any integer $n \equiv 2 \pmod 8$ that is not a sum of two squares is not a sum of less than $10$ squares of odd integers. Therefore, we have $S(2)=10$.
\end{proof}

\begin{prop}
Every positive integer is a sum of at most $6$ squares not divisible by $3$, and in fact, $S(3)=6$.
\end{prop}
\begin{proof}
Let $n$ be a positive integer. First, assume that $n\equiv 1 \pmod 3$. By Lagrange's four-square theorem, $n$ is a sum of four squares, that is, there are integers $a_1,a_2,a_3$, and $a_4$ such that $n=a_1^2+a_2^2+a_3^2+a_4^2$. If $a_1a_2a_3a_4$ is not divisible by $3$, then 
$n \stackrel{3}{\lra} I_4$. If $a_1a_2a_3a_4$ is divisible by $3$, then exactly three of $a_1,a_2,a_3$, and $a_4$ are divisible by $3$. Without loss of generality, we assume that $a_1,a_2$, and $a_3$ are divisible by $3$.  Since $n \stackrel{3}{\lra} I_1$ in the case when $a_1=a_2=a_3=0$, we assume that $a_1^2+a_2^2+a_3^2 \ne 0$.  
 By applying Lemma \ref{class1} in the case when $f=\langle1,1,1\rangle$ and $p=3$, there are integers $b_1,b_2$, and $b_3$ such that
$$
a_1^2+a_2^2+a_3^2=b_1^2+b_2^2+b_3^2 \quad \text{and} \quad (3,b_1b_2b_3)=1.
$$
In fact, Lemma \ref{class1} says that at least one of $b_1,b_2$, and $b_3$ is not divisible by $3$. However, in our case, this implies that none of $b_i$'s are divisible by $3$. Hence if $n\equiv 1 \pmod 3$, then $n \stackrel{3}{\lra} I_1$ or $n \stackrel{3}{\lra} I_4$. 

Now, assume that $n\equiv 2 \pmod 3$. If $n$ is a sum of two squares, then $n \stackrel{3}{\lra} I_2$. Otherwise, we have $n \stackrel{3}{\lra} I_5$, for $n-1\equiv 1 \pmod 3$. Finally assume that $n\equiv 0 \pmod 3$. In this case, we have $n \stackrel{3}{\lra} I_3$ or $n \stackrel{3}{\lra} I_6$. Note that if $n$ is not a sum of three squares, then $n$ is not a sum of less than or equal to $5$ squares not divisible by $3$. Therefore, we have $S(3)=6$.
\end{proof}

\section{When $n$ is divisible by  $p$}

In this and next section, we find $S(p)$ for a prime $p$ greater than $3$. 
 In this section, we find the smallest number $k$ to represent a positive integer $n$ divisible by $p$ as a sum of less than or equal to $k$ squares not divisible by $p$. 

\begin{lem}\label{primitive}
Let $p$ be an odd prime and let $n$ be a positive integer. Assume that $p$ is represented by $\langle1,k\rangle$, where $k$ is a positive integer not divisible by $p$. If an integer $n$ divisible by $p$ is represented by $\langle1,k\rangle$, then there are integers $u$ and $v$ such that
$$
n=u^2+kv^2 \quad \text{and} \quad (p,uv)=1.
$$
\end{lem}
\begin{proof}
See \cite {oh2}.
\end{proof}

\begin{lem}\label{three1}
Let $p \equiv 1 \pmod 4$ be a prime  and let $n$ be a positive integer. If $n$ is a sum of three squares, then $n$ is a sum of $k$ squares not divisible by $p$ for some integer $k\le 4$.
\end{lem}
\begin{proof}
From the assumption, there are integers $a_1,a_2$, and $a_3$ such that $n=a_1^2+a_2^2+a_3^2$. First, assume that exactly two of $a_1, a_2$, and $a_3$ are divisible by $p$. Without loss of generality, assume that both $a_1$ and $a_2$ are divisible by $p$. If $a_1=a_2=0$, then $n \stackrel{p}{\lra} I_1$.  If $a_1^2+a_2^2\ne 0$, then by Lemma \ref{primitive}, there are integers $b_1$ and $b_2$ such that
$$
a_1^2+a_2^2=b_1^2+b_2^2 \quad \text{and} \quad (p,b_1b_2)=1.
$$
Therefore, we have $n \stackrel{p}{\lra} I_3$. Next, assume that exactly one of $a_1, a_2$, and $a_3$, say $a_1$, is divisible by $p$. If $a_1=0$, then $n \stackrel{p}{\lra} I_2$. If $a_1\ne0$, then by Lemma \ref{primitive}, there are integers $c_1$ and $c_2$ such that
$$
a_1^2=c_1^2+c_2^2 \quad \text{and} \quad (p,c_1c_2)=1.
$$
Hence we have $n \stackrel{p}{\lra} I_4$. Finally, assume that  $a_i$ is divisible by $p$ for any $i=1,2,3$. In this case, from the above assertion, we may easily show that $n \stackrel{p}{\lra} I_k$ for some integer $k\le 4$. This completes the proof.
\end{proof}

\begin{prop}\label{divisible1}
Let $p \equiv 1 \pmod 4$ be a prime  and let $n$ be a positive integer.  If $n$ is divisible by $p$, then $n \stackrel{p}{\lra} I_k$ for some integer $k\le 4$.
\end{prop}

\begin{proof} Without loss of generality, we may assume that $\ord_q(n)\le 1$ for any prime $q\ne p$. 
By Lemma \ref{three1}, we may assume that $n\equiv 7 \pmod 8$. Since the class number of $\langle1,1,5\rangle$ is one and every positive integer congruent to $7$ modulo $8$ is represented by $\langle1,1,5\rangle$ over $\z_p$ for any prime $p$, there are integers $x,y$, and $z$ such that $n=x^2+y^2+5z^2$ by 102.5 of \cite{om}. If $xyz$ is not divisible by $p$, then $n=x^2+y^2+z^2+(2z)^2$ and $n \stackrel{p}{\lra} I_4$. Assume that at least two of $x,y$, and $z$ are divisible by $p$. Then, both $x$ and $y$ are divisible by $p$. If $x^2+y^2 \ne 0$, then there are integers $a$ and $b$  not divisible by $p$ such that $x^2+y^2=a^2+b^2$. If $z$ is not divisible by $p$, then $p=5$. Since $5z^2=(2z)^2+z^2$, we have $n \stackrel{p}{\lra} I_2$ or $n \stackrel{p}{\lra} I_4$. Assume that $z$ is a non-zero integer divisible by $p$. Then there are integers $c$ and $d$  not divisible by $p$ such that $z^2=c^2+d^2$. Hence we have 
$$
5z^2=(2c+d)^2+(c-2d)^2=(2c-d)^2+(c+2d)^2.
$$
Now, one may easily check that either $(2c+d)(c-2d)$ or $(2c-d)(c+2d)$ is not divisible by $p$.  
Therefore, we have  $n \stackrel{p}{\lra} I_2$ or $n \stackrel{p}{\lra} I_4$. 

Now, assume that $x$ is divisible by $p$ and $yz$ is not divisible by $p$. Since $y^2+5z^2\equiv 0 \pmod p$, we have $\left(\frac{-5}{p}\right)=\left(\frac{5}{p}\right)=1$. Let $x=p^tx'$ with $(p,x')=1$. Since the class number of $\langle1,5\rangle$ is one, both $p$ and $p^{2t}$ is represented by $\langle1,5\rangle$. Hence by Lemma \ref{primitive}, there are integers $u$ and $v$ such that $p^{2t}=u^2+5v^2$ and $(p,uv)=1$. Then we have
$$
\begin{array}{ll}
n=x^2+y^2+5z^2&=p^{2t}{x'}^2+y^2+5z^2\\
              &=(u^2+5v^2) {x'}^2+y^2+5z^2 \\     
              &= (vx'+2z)^2+(2vx'-z)^2+u^2{x'}^2+y^2\\
              &= (vx'-2z)^2+(2vx'+z)^2+u^2{x'}^2+y^2,
\end{array}
$$
where $uvx'yz$ is not divisible by $p$. Note that either $(vx'+2z)(2vx'-z)$ or  $(vx'-2z)(2vx'+z)$ is not divisible by $p$. Hence $n \stackrel{p}{\lra} I_k$ for some integer $3\le k\le4$. The proof of the case when $y$ is divisible by $p$ and $xz$ is not divisible by $p$ is quite similar to this. If $z$ is divisible by $p$ and $xy$ is not divisible by $p$, then one may easily show that $n \stackrel{p}{\lra} I_2$ or $n \stackrel{p}{\lra} I_4$ by the similar reasoning given above.
\end{proof}

Now, we consider the case when $p \equiv 3 \pmod 4$ and $n$ is divisible by $p$. To deal with this case, we need some results from the theory of modular forms. For general theory of modular forms and some relation between representations of quadratic forms and modular forms, see \cite{ono} and \cite{wp}.

For a positive integer $N$ and a positive rational number $k$ such that $2k \in \z$, let $S_{k}(N,\chi)$ be the space of cusp forms of weight $k$ with character $\chi$ for the congruence group $\Gamma_0(N)$.

\begin{lem}\label{ramanujan}
Let $f=\langle 1,1,10\rangle$ be Ramanujan's ternary quadratic form and let $n$ be a positive integer. For any prime $p \ne 2,3,5$, and $17$, we have
$$
r(p^2n,f)-r(n,f)>0,
$$
provided that $p^2n$ is represented by $f$.
\end{lem}
\begin{proof}
Note that $h(f)=2$ and
$$
\gen(f)/\sim=\left\{f, f'=\langle2\rangle \perp \begin{pmatrix}2&1 \\ 1&3\end{pmatrix} \right\}.
$$
We let 
$$
\phi(z)=\displaystyle\sum_{n=1}^{\infty}a(n)q^n=\displaystyle\frac14\sum_{n=1}^{\infty}(r(n,f)-r(n,f'))q^n=\displaystyle q-q^3-q^7-q^9+2q^{13}+\cdots,
$$
where $q=e^{2\pi i z}$. Then it is known that $\phi(z)\in S_{\frac32}\left(40,\left(\frac{10}{.}\right)\right)$ is the weight $\frac32$ cusp form. It is also known  (see, for example, \cite{os}) that the Shimura lift of $\phi(z)$ is a cusp form of  weight $2$ 
$$
\begin{array}{ll}
\Phi(z)&\!\!\!\!=\eta^{2}(2z)\eta^{2}(10z)=\displaystyle\sum_{n=1}^{\infty}A(n)q^n\\ [1em]
                              &\!\!\!\!=q-2q^3-q^5+2q^7+q^9+2q^{13}+2q^{15}-6q^{17}-4q^{19}-4q^{21}+6q^{23}+\cdots.
\end{array}
$$
Here $\eta(z)=q^{\frac{1}{24}}\prod_{n=1}^{\infty}(1-q^n)$ is the Dedekind's eta-function. Since the dimension of the space $S_{\frac32}\left(40,\left(\frac{10}{.}\right)\right)$ is one, $\phi(z)$ is an eigenform of all Hecke operators $T(p^2)$. 
Hence for any prime $p$, there is a complex number $\alpha(p)$ such that
\begin{equation}\label{hecke}
\alpha(p)a(n)=a(p^2n)+\left(\frac{-10n}{p}\right)a(n)+\left(\frac{10}{p}\right)^2\cdot p\cdot a(n/p^2),
\end{equation}
for any positive integer $n$. Here $a(n/p^2)=0$ if $n$ is not divisible by $p^2$. Since $\Phi(z)\in S_2(20)$ is the newform and the Shimura lifts commute with the Hecke operators of integral and half-integral weight, we have $\alpha(p)=A(p)$ for any prime $p$. By Deligne's bound on Hecke eigenvalues, we have $\vert A(n) \vert \le \tau(n)n^{\frac12}$, where $\tau(n)$ is the number of positive divisors of $n$. Hence we have $\vert \alpha(p) \vert=\vert A(p) \vert \le 2\sqrt{p}$ for any prime $p$. 

Let $p$ be a prime relatively prime to $10$ and let $\ord_p(n)=\lambda_p$. From the assumption given above, we may assume that $n$ is represented by $f$. Then, by Equation \eqref{hecke}, we have
$$
\begin{array}{ll}
r(p^2n,f)-r(p^2n,f')&=\left(\alpha(p)-\left(\frac{-10n}{p}\right)\right)(r(n,f)-r(n,f'))\\
                    &\hspace{0.4cm}-\, p\displaystyle \cdot\left(r\left(\frac n{p^2},f\right)-r\left(\frac n{p^2},f'\right)\right).
\end{array}
$$
By Lemma \ref{t-square}, we also have
$$
\begin{array}{ll}
r(p^2n,f)+2r(p^2n,f')&=\displaystyle \left( \frac{ \displaystyle p^{[\frac{\lambda_p}{2}] +2}-1-\left(\frac{-np^{-2[\frac{\lambda_p}{2}]}\cdot df}{p}\right)(p^{[\frac{\lambda_p}{2}] +1}-1)}{ \displaystyle p^{[\frac{\lambda_p}{2}] +1}-1-\left(\frac{-np^{-2[\frac{\lambda_p}{2}]}\cdot df}{p}\right)(p^{[\frac{\lambda_p}{2}]}-1)} \right)\\
                    &\hspace{0.4cm}\times (r(n,f)+2r(n,f'))\\[0.3cm]
                    &\hspace{-3cm}=\displaystyle \left( \frac{ \displaystyle p^{[\frac{\lambda_p}{2}] +1}(p-2)+2p-1-\left(\frac{-np^{-2[\frac{\lambda_p}{2}]}\cdot df}{p}\right)(p^{[\frac{\lambda_p}{2}]}(p-2)+2p-1)}{ \displaystyle p^{[\frac{\lambda_p}{2}] +1}-1-\left(\frac{-np^{-2[\frac{\lambda_p}{2}]}\cdot df}{p}\right)(p^{[\frac{\lambda_p}{2}]}-1)} \right)\\ [3em]
                    &\hspace{-2.6cm}\times \displaystyle (r(n,f)+2r(n,f'))+2p\cdot\left(r\left(\frac n{p^2},f\right)+2r\left(\frac n{p^2},f'\right)\right).
                    
\end{array}                    
$$
By combining two equalities given above, we have
$$
\begin{array}{ll}
3(r(p^2n,f)-r(n,f))&\!\!\!\ge\left(p-5+2\alpha(p)-2\left(\frac{-10n}{p}\right)\right)r(n,f)\\
                   &\!\!+\left(2p-4-2\alpha(p)+2\left(\frac{-10n}{p}\right)\right)r(n,f')+6p\cdot r\displaystyle \left(\frac n{p^2},f'\right).
\end{array}
$$
If $p\ne 3,17$, then $p-5+2\alpha(p)-2\left(\frac{-10n}{p}\right)>0$ and 
$$2
p-4-2\alpha(p)+2\left(\frac{-10n}{p}\right)>0,
$$ 
for any prime $p$. Therefore if $p\ne 2,3,5$, and $17$, we have  $r(p^2n,f)-r(n,f)>0$.
This completes the proof. \end{proof}

The following proposition is mentioned in Remark 3.2 of \cite{oh1} without proof. Here, we provide a simple proof for those who are unfamiliar with the method developed in \cite{oh1}.  

\begin{prop} \label{ramatec} For any positive integer $n$ such that $n\equiv 5\pmod 6$, the diophantine equation $n=x^2+y^2+10z^2$ has always an integer solution. 
\end{prop}

\begin{proof} One may easily show that every integer $n$ such that $n\equiv 5\pmod 6$ is represented by the genus of Ramanujan's ternary quadratic form $f=\langle 1,1,10\rangle$. Hence we may assume that $n$ is represented by the other ternary quadratic form in the genus of Ramanujan's ternary quadratic form, that is, there are integers $a,b$, and $c$ such that $n=2a^2+2b^2+2bc+3c^2$. Then, by a direct computation, we have
$$
\begin{array}{ll}
&(a,b,c)\equiv (0,\pm1,0), (\pm1,0,0), (1,0,\pm1), \\ [0.5em]
&\hspace{3.15cm}(1,\pm1,\mp1), (-1,0,\pm1) \ \text{or} \ (-1,\pm1,\mp1) \pmod 3.
\end{array}
$$
By changing signs of $a,b$, and $c$, if necessary, we may assume that  
$$
(a,b,c)\equiv (0,1,0), (1,0,0), (1,0,1) \ \text{or} \ (1,1,-1) \pmod 3.
$$
First, assume that $(a,b,c)\equiv(1,1,-1) \pmod 3$. Then there are integers $b'$ and $c'$ such that $b-a=3b'$ and $c-2a=3c'$. Therefore we have
$$
\begin{array}{rl}
n\!\!\!&=2a^2+2(3b'+a)^2+2(3b'+a)(3c'+2a)+3(3c'+2a)^2\\
 &=20a^2+18{b'}^2+27{c'}^2+24ab'+42ac'+18b'c'\\
 &=(a+3b'-c')^2+(3a+3b'+4c')^2+10(a+c')^2,
\end{array}
$$
which implies that $n$ is represented by $\langle 1,1,10 \rangle$. Similarly, one may easily check that  $n$ is represented by $\langle 1,1,10 \rangle$ in the cases when $(a,b,c)\equiv (0,1,0) \pmod 3$ or $(a,b,c)\equiv (1,0,1) \pmod 3$.

   Finally, assume that $(a,b,c)\equiv (1,0,0) \pmod 3$. Let $\mathfrak G_{-20}$ be the set of all proper classes of primitive binary quadratic forms with discriminant $-20$. Then it is well known that $\mathfrak G_{-20}$ forms an abelian group with the composition law (for details, see \cite{ca}). In fact, $\mathfrak G_{-20}=\{[x^2+5y^2], [2x^2+2xy+3y^2]\}$ and $[x^2+5y^2]$ is the identity class. Since $3^2$ is primitively represented by the identity class $[x^2+5y^2]$, every integer that is represented by $2x^2+2xy+3y^2$ is $3$-primitively represented by it by Theorem 4.1 of \cite{oh2}. This implies that there are always integers $d,e$ such that 
$$
2b^2+2bc+3c^2=2{d}^2+2de+3{e}^2\quad \text{and} \quad (3,d,e)=1,
$$
 unless $b=c=0$. This implies that $(a,d,e)\equiv (1,\pm1,\mp1)$ or $(1,0,\pm1) \pmod 3$. Therefore $n$ is represented by $\langle 1,1,10 \rangle$ from the above argument.  Note that any integer of the form $2a^2$ is represented by Ramanujan's ternary quadratic form $\langle 1,1,10\rangle$. This completes the proof. 
\end{proof}

\begin{lem} \label{tec}  Let $p\equiv 3 \pmod 4$ be a prime and let $n$ and $a$ be  positive integers such that 
\begin{itemize} 
\item [(i)] the diagonal ternary quadratic form $\langle 1,1,a\rangle$ has class number $1$;
\item [(ii)] the integer $a$ is either a square or a sum of $2$ squares not divisible by $p$;
\item [(iii)] $\left(\frac ap\right)=1$ and $n$ is divisible by $p$;
\item [(iv)] $n$ is represented by $\langle 1,1,a\rangle$.
\end{itemize}
Then, $n \stackrel{p}{\lra} I_3$ if $a$ is a square, and $n \stackrel{p}{\lra} I_4$ if $a$ is a sum of $2$ squares not divisible by $p$.
\end{lem}

\begin{proof} By condition (iv), there are integers $x_1,x_2$, and $x_3$ such that $n=x_1^2+x_2^2+ax_3^2$. 
Since the proofs are quite similar to each other, we assume that  $a$ is a sum of $2$ squares not divisible by $p$.
Then, there are integers $a_1$ and $a_2$ such that $a=a_1^2+a_2^2$ and $(p,a_1a_2)=1$. Hence we have
$$
n=x_1^2+x_2^2+(a_1x_3)^2+(a_2x_3)^2.
$$
If one of $x_i$'s is divisible by $p$, then by condition (iii), all of the $x_i$'s are divisible by $p$. Therefore, if $\ord_p(n)=1$, then $n \stackrel{p}{\lra} I_4$. If $\ord_p(n) \ge 2$, then we may take $x_i$'s such that at least one of them is not divisible by $p$ by Lemma \ref{class1}. This implies that all of the $x_i$'s are not divisible by $p$. 
This completes the proof.  \end{proof}

\begin{rmk} \label{remark1} Note that  the class number of Ramanujan's ternary quadratic form $\langle 1,1,10\rangle$ is two. 
However, we may apply the above lemma to the case when $a=10$ and $p>3$ by using Lemma \ref{ramanujan} instead of Lemma \ref{class1}. In this case, it is not easy to check whether condition (iv) of Lemma \ref{tec} holds or not without Generalized Riemann Hypothesis(GRH) (see \cite{os}).    
\end{rmk}

\begin{prop}\label{divisible2}
Let $p\equiv 3\pmod 4$ be a prime greater than $3$ and let $n$ be a positive integer divisible by $p$. Then $n \stackrel{p}{\lra} I_k$ for some integer $k=3$ or $4$. 
\end{prop}

\begin{proof} We may assume, without loss of generality, that $\ord_q(n) \le 1$ for any prime $q \ne p$. 
If $n$ is a sum of three squares, then $n \stackrel{p}{\lra} I_3$ by Lemma \ref{tec}. Assume that $n \equiv 7 \pmod 8$. Then $n$ is represented by both $\langle 1,1,2\rangle$ and $\langle 1,1,5\rangle$. If  either $\left(\frac{2}{p}\right)=1$ or $\left(\frac{5}{p}\right)=1$, then $n \stackrel{p}{\lra} I_4$ by Lemma \ref{tec}. 

Now, assume that $\left(\frac{2}{p}\right)=\left(\frac{5}{p}\right)=-1$. Then $\left(\frac{10}{p}\right)=1$. Assume further that $\ord_p(n)\ge 2$. Then $n$ is represented by Ramanujan's ternary quadratic form $\langle 1,1,10\rangle$ (see Theorem 1 of \cite {os}). Hence we may still apply Lemma \ref{tec} to show that  $n \stackrel{p}{\lra} I_4$, as stated in Remark \ref{remark1}. 

Finally, assume that $\ord_p(n)=1$. If $n \equiv 2 \pmod 3$, that is, $n\equiv 5 \pmod 6$, then $n$ is represented by $\langle 1,1,10 \rangle$ by Proposition \ref{ramatec}, which implies that $n \stackrel{p}{\lra} I_4$ by Lemma \ref{tec}.
Next, assume that $n\equiv 1 \pmod 3$. Then exactly one of $n-1, n-4$, and $n-25$ is divisible by $9$. Let $s_0\in\{1,2,5\}$ be the integer such that $n-(s_0)^2\equiv 0 \pmod 9$. 
Since $n-(s_0)^2\equiv 3$ or $6 \pmod 8$, the integer $\frac{n-(s_0)^2}{9}$ is a sum of three squares, whereas it is not a sum of two squares. Hence there are non-zero integers $a,b$, and $c$ such that
$$
n-(s_0)^2=(\pm a)^2+(\pm b)^2+(\pm c)^2\quad \text{and}\quad a\equiv b \equiv c\equiv 0 \pmod 3.
$$
Since $n-(s_0)^2$ is not divisible by $p$, at least one of $a,b$, and $c$ is not divisible by $p$. If $abc$ is not divisible by $p$, then $n \stackrel{p}{\lra} I_4$. Hence we may assume that $abc$ is divisible by $p$. Since $-1$ is not a square modulo $p$, we may assume that exactly one of $a,b$, and $c$ is divisible by $p$. Without loss of generality, assume that $c$ is divisible by $p$. By choosing signs suitably, we further assume that 
\begin{equation}\label{nequiv}
a\nequiv 2b \pmod p, \quad   2a\nequiv b \pmod p, \quad \text{and}\quad a\nequiv -b \pmod p.
\end{equation}
Now, we have
$$
n-(s_0)^2=(2m-a)^2+(2m-b)^2+(2m-c)^2 \quad \text{and}\quad a+b+c=3m,
$$
where $m$ is an integer. By \eqref{nequiv}, the integer $(2m-a)(2m-b)(2m-c)$ is not divisible by $p$, which implies that $n \stackrel{p}{\lra} I_4$. 

Now, assume that $n \equiv 0 \pmod 3$. Since $\frac{n}{3}\equiv 5 \pmod 8$, we have $\frac{n}{3} \stackrel{p}{\lra} I_3$ by Lemma \ref{tec}. Hence there are integers $a,b$, and $c$ such that
$$
n=3(a^2+b^2+c^2) \quad \text{and} \quad (p,abc)=1.
$$
By Euler's four-square identity, we have
$$
\begin{array}{rl}
n\!\!\!&=(1^2+1^2+1^2+0^2)(a^2+b^2+c^2+0^2) \\
 &=(a-b-c)^2+(a+b)^2+(a+c)^2+(b-c)^2 \\
 &=(a+b+c)^2+(a-b)^2+(a-c)^2+(b-c)^2 \\
 &=(a+b-c)^2+(a-b)^2+(a+c)^2+(b+c)^2 \\
 &=(a-b+c)^2+(a+b)^2+(a-c)^2+(b+c)^2.
\end{array}
$$
Assume that $a-b-c\equiv 0 \pmod p$. Then clearly, $(a+b+c)(a-b)(a-c)$ is not divisible by $p$. If $b-c$ is divisible by $p$, then $a\equiv 2b \pmod p$. This implies that  $n=a^2+b^2+c^2 \equiv 6b^2 \nequiv 0 \pmod p$, 
which is a contradiction to the fact that $n\equiv 0\pmod p$. Hence $b-c$ is not divisible by $p$ and $n \stackrel{p}{\lra} I_4$. Similarly, if  one of $(a+b+c), (a+b-c)$, and $(a-b+c)$ is divisible by $p$, then $n \stackrel{p}{\lra} I_4$. Therefore we may assume that 
$$
(a-b-c)(a+b+c)(a+b-c)(a-b+c) \not \equiv 0 \pmod p.
$$ 
Since $a$ is not divisible by $p$, we have 
\begin{equation} \label{final}
a+b\nequiv 0 \pmod p \quad \text{or} \quad a-b\nequiv 0 \pmod p.
\end{equation}
 If both $(a+c)(b-c)$ and $(a-c)(b+c)$ are divisible by $p$, then 
$$
a\equiv b \equiv -c \pmod p\quad \text{or} \quad  a\equiv b \equiv c \pmod p,
$$
which implies that $n=3(a^2+b^2+c^2)\equiv 9c^2 \not  \equiv 0 \pmod p$. This is a contradiction. Similarly, one may easily show that either
$(a-c)(b-c)$ or $(a+c)(b+c)$ is not divisible by $p$. From these and Equation \eqref{final}, we have $n \stackrel{p}{\lra} I_4$. 
\end{proof}

\section{When $n$ is not divisible by $p$}

In this section, we consider the case when a positive integer $n$ is not divisible by $p$, where $p$ is a prime greater than $3$, as in the previous section. In this case, we may assume that $n$ is square-free.  

\begin{lem} \label {calculation}
Let $n$ be a positive integer and let $p$ be a prime greater than $3$.
\begin{itemize}
\item [(i)] If $5\le p\le 13$ and $n<(16p)^2$, then $n \stackrel{p}{\lra} I_k$ for some integer $k\le 4$, except the case when $p=5$ and $n$ is $79$. In fact, $79$ is a sum of $5$ squares not divisible by $5$.
\item [(ii)]If $ p \ge 17$ and $n<(10p)^2$, then $n \stackrel{p}{\lra} I_k$ for some integer $k\le 4$.
\end{itemize}
\end{lem}
\begin{proof}
To prove that $n$ is a sum of $k$ squares not divisible by $p$ for some positive integer $k$, we may assume that $0\le\ord_q(n)\le 1$ for any prime $q\ne p$. By Lagrange's four square theorem, we  may also assume that $n\ge p^2$.

First, assume that $p^2 \le n < (10p)^2$. Then there are integers  $u \ (1\le u\le 9)$ and $a \ (0\le a \le p-1)$ such that $(up+a)^2\le n < (up+a+1)^2$. We assume that $p\ge 73$. 
If $a=0$, then
$$
n-(up-1)^2 \le n-(up-2)^2 < (up+1)^2-(up-2)^2 \le 6up-3<p^2.
$$ 
If $a=1$, then
$$
n-(up+1)^2 \le n-(up-2)^2 < (up+2)^2-(up-2)^2 \le 8up<p^2,
$$ 
and if $a \ge 2$, then
$$
n-(up+a)^2 \le n-(up+a-1)^2 < (up+a+1)^2-(up+a-1)^2 \le 4up+4a<p^2.
$$
Therefore, there is a positive integer $k$ less than or equal to $3$ such that
$$
\begin{cases}
\text{$n-(up-1)^2$ or $n-(up-2)^2$} \stackrel{p}{\lra} I_k \ &\text{if $a=0$},\\
\text{$n-(up-2)^2$  or $n-(up+1)^2$} \stackrel{p}{\lra} I_k \ &\text{if $a=1$},\\
\text{$n-(up+a-1)^2$ or $n-(up+a)^2$} \stackrel{p}{\lra} I_k \ &\text{if $a\ge 2$}.
\end{cases}
$$
Therefore if $p\ge 73$ and $n<(10p)^2$, then $n \stackrel{p}{\lra} I_k$ for some integer $k\le 4$.

For the case when $5\le p \le 71$, one may check by a direct computation that $n \stackrel{p}{\lra} I_k$ for some integer $k \le 4$, except the case when $p=5$ and $n=79$. Note that essentially different representations of $79$ by $I_4$  are  
$$
79=1^2+2^2+5^2+7^2=2^2+5^2+5^2+5^2=3^2+3^2+5^2+6^2.
$$  
Hence $79$ is not represented by a sum of $4$ squares not divisible by $5$. Since $79=1^2+1^2+2^2+3^2+8^2$, $79$ is a sum of $5$ squares not divisible by $5$.   
\end{proof}


\begin{thm}\label{0or2}
If $n\equiv 0$ or $2\pmod 3$, then $n \stackrel{p}{\lra} I_k$ for some integer $k\le 4$.
\end{thm}
\begin{proof} 
Reacll that we are assuming that $p$ is a prime greater than $3$ and $n$ is a square-free positive integer not divisible by $p$. By Lemma \ref{calculation}, we may further assume that 
\begin{equation} \label{cond}
\begin{cases} n\ge (16p)^2 \quad &\text{if $5\le p\le 13$,}\\
             n\ge (10p)^2 \quad &\text{if $p \ge 17$}.\\
             \end{cases}  
\end{equation}
First, we will prove that there exist integers $k$ and $s_0\in\{0,1,2,3\}$ such that 
\begin{enumerate}
\item $\displaystyle n-(6k+s_0)^2\ne 4^{\alpha}(8\beta+7)$ for any non-negative integers $\alpha$ and $\beta$;
\item $\displaystyle n-(6k+s_0)^2\equiv 2 \pmod 3$;
\item $\displaystyle n-(6k+s_0)^2>0$;
\item $\displaystyle\left(\frac{n-(6k+s_0)^2}{p}\right)\ne\left(\frac{5}{p}\right), 0$;
\item $\displaystyle 6k+s_0\nequiv 0\pmod p$.
\end{enumerate}
We choose an integer $s_0$ such that 
$$
s_0=\begin{cases} 0 \quad &\text{if $n \not \equiv 3 \pmod 4$ and $n\not \equiv 0 \pmod 3$,}\\
                   1 \quad &\text{if $n \not \equiv 1 \pmod 4$ and $n\equiv 0 \pmod 3$,}\\
                   2 \quad &\text{if $n  \equiv 1 \pmod 4$ and $n\equiv 0 \pmod 3$,}\\
                   3 \quad &\text{if $n  \equiv 3 \pmod 4$ and $n\not \equiv 0 \pmod 3$.}\\ 
\end{cases}
$$
Then clearly, the first and the second conditions hold for any integer $k$. Now, we will find an integer $k$ satisfying the above conditions (3)$\sim$(5) for this integer $s_0$. 
If $p$ is $5$ or $7$, then one may easily find an integer $k$ such that the above conditions (3)$\sim$(5) are all satisfied. Hence we may assume that $p$ is greater than $7$. In fact, we will choose an integer $k$ in the set $T=\{1,2,\cdots,\frac{p+9}{2}\}$. Since $n$ satisfies \eqref{cond}, we have $n-(6k+s_0)^2>0$ for any $k \in T$ and any $s_0\in \{0,1,2,3\}$. It is well known that the number of solutions $(x,y)$ of the equation $x^2+y^2=n$ over $\mathbb F_p$ is greater than or equal to $p-1$. Hence the number of $x_0$'s such that $n-x_0^2$ is a zero or a square in $\mathbb F_p$ is at least $\frac {p-1}2$. Since $\vert T \vert=\frac{p+9}{2}$, there are at least four $k\in T$ such that 
$n-(6k+s_0)^2$ is a zero or a square in $\mathbb F_p$. Similarly, there are at least four $k\in T$ such that $n-(6k+s_0)^2$ is a zero or a non-square in $\mathbb F_p$. Therefore there exists a positive integer $k\in T$ such that
$$
\displaystyle\left(\frac{n-(6k+s_0)^2}{p}\right)\ne\left(\frac{5}{p}\right), 0 \quad \text{and}\quad \displaystyle6k+s_0\nequiv 0\pmod p.
$$ 

Now, by (1) and (3), there are integers $a,b$, and $c$ such that 
\begin{equation} \label{signs}
n-(6k+s_0)^2=(\pm a)^2+(\pm b)^2+(\pm c)^2.
\end{equation}
Since $n-(6k+s_0)^2\equiv 2 \pmod 3$, we may suitably choose signs in Equation \eqref{signs} so that $a+b+c$ is divisible by $3$. 
If  $a+b+c=3m$ for some integer $m$, then we have  
\begin{equation} \label{hard}
n-(6k+s_0)^2=a^2+b^2+c^2=(2m-a)^2+(2m-b)^2+(2m-c)^2.
\end{equation}
 If $abc$ is not divisible by $p$, then $n \stackrel{p}{\lra} I_4$. Now, assume that at least two of $a,b$, and $c$ are divisible by $p$. Without loss of generality, we assume that both $a$ and $b$ are divisible by $p$. Since $c$ is not divisible by $p$  by $(4)$, $m$ is not divisible by $p$. Therefore, $(2m-a)(2m-b)(2m-c)$ is not divisible by $p$. This implies that $n \stackrel{p}{\lra} I_4$. 
 
 Finally, assume that exactly one of $a,b$, and $c$ is divisible by $p$. Without loss of generality, we assume that $c$ is divisible by $p$ and $ab$ is not divisible by $p$. We will show that there are integers $a,b$, and $c$ satisfying \eqref{hard} such that $m$ is not divisible by $p$. Suppose, on the contrary, that $m$ is divisible by $p$.  Since $n-(6k+s_0)^2\equiv 2 \pmod 3$, exactly one of $a,b$, and $c$ is divisible by $3$. If $a$ is divisible by $3$, then 
 $$
 n-(6k+s_0)^2=(-a)^2+b^2+c^2 \quad \text{and} \quad -a+b+c=3\left(m-\frac23a\right),
 $$
  where $(m-\frac23a)$ is not divisible by $p$. The same argument can be applied to the case when $b$ is divisible by $3$.  Hence we may assume that $c$ is divisible by $3$.  Since $a+b\equiv 0 \pmod 3$, there are integers $b_1$ and $c_1$ such that  $a+b=3b_1$ and $c=3c_1$. Then, we have
$$
\begin{array}{ll}
n-(6k+s_0)^2 &=a^2+(-a+3b_1)^2+(3c_1)^2\\
  &=(a-2b_1+2c_1)^2+(-a+b_1+2c_1)^2+(-2b_1-c_1)^2\\
  &=(a-2b_1-2c_1)^2+(-a+b_1-2c_1)^2+(-2b_1+c_1)^2.
\end{array}
$$
Note that both $b_1$ and $c_1$ are divisible by $p$. By applying the similar argument given above to this situation, we may assume that $-2b_1+c_1\equiv 0 \pmod 3$ and $-2b_1-c_1\equiv 0 \pmod 3$. This implies that $b_1\equiv c_1\equiv 0 \pmod 3$. Now,  suppose that $t$, $b_t$, and $c_t$ are integers such that 
$$
a+b=3^t b_t,  \  \  c=3^tc_t, \quad \text{and} \quad \text{ either $b_t$ or $c_t$ is not divisible by $3$.}\ 
$$
 Note that both $b_t$ and $c_t$ are divisible by $p$.  Let $x_t, y_t$ be integers such that 
 $$
 x_t^2+2y_t^2=3^{2t} \quad \text{and} \quad x_ty_t \not \equiv 0 \pmod 3.
 $$
 Note that such integers always exist by Lemma \ref{primitive}. Then we have
$$
\begin{array}{rl}
\!\!n-(6k+s_0)^2&\!\!\!\!\!=\!a^2+(-a+3^tb_t)^2+(3^tc_t)^2\\
&\!\!\!\!\!=\!\left(a+\frac{(x_t-3^t)}{2}b_t+y_tc_t\right)^2\!\!\!+\!\left(-a+\frac{(x_t+3^t)}{2}b_t+y_tc_t\right)^2\!\!\!+\!(y_tb_t-x_tc_t)^2\\
&\!\!\!\!\!=\!\left(a+\frac{(x_t-3^t)}{2}b_t-y_tc_t\right)^2\!\!\!+\!\left(-a+\frac{(x_t+3^t)}{2}b_t-y_tc_t\right)^2\!\!\!+\!(y_tb_t+x_tc_t)^2.
\end{array}
$$
Now, by applying the same argument given above, we may assume that  
$$
y_{t}b_{t}-x_{t}c_{t} \equiv 0 \pmod 3 \quad  \text{and}  \quad  y_{t}b_{t}+x_{t}c_{t} \equiv 0 \pmod 3,
$$
which implies that $b_t \equiv c_t \equiv 0 \pmod 3$. This is a contradiction. Therefore, we may assume that the integer $m$ given in \eqref{hard} is not divisible by $p$. 
Since $c$ is divisible by $p$  and $abm$ is not divisible by $p$, $2m-c$ is not divisible by $p$. If $2m-a$ is divisible by $p$, then $a\equiv 2m\pmod p$ and $b\equiv m\pmod p$. This implies that $n-(6k+s_0)^2\equiv 5m^2\pmod p$, which is a contradiction to the fact that $\left(\frac{n-(6k+s_0)^2}{p}\right)\ne\left(\frac{5}{p}\right)$. Therefore $2m-a$ is not divisible by $p$. By  similar reasoning, $2m-b$ is not divisible by $p$. Therefore, by \eqref{hard},  we have $n \stackrel{p}{\lra} I_4$. This completes the proof. 
\end{proof}

\begin{thm}\label{1}
If $n\equiv 1\pmod 3$, then $n \stackrel{p}{\lra} I_k$ for some integer $k\le 4$, except the case when $p=5$ and $n=79$. 
\end{thm}
\begin{proof}
Recall that we are assuming that $p$ is a prime greater than $3$ and $n$ is a square-free positive integer not divisible by $p$. By Lemma \ref{calculation}, we may further assume that \eqref{cond} holds. 
Then, similarly to Theorem \ref{0or2}, one may easily show that there exist integers $k$ and $s_0\in\{1,2,4,5,7,8\}$ such that
\begin{enumerate}
\item $\displaystyle n-(18k+s_0)^2 \ne 4^{\alpha}(8\beta+7)$ for any integers $\alpha$ and $\beta$;
\item $\displaystyle n-(18k+s_0)^2\equiv 0 \pmod 9$;
\item $\displaystyle n-(18k+s_0)^2>0$; 
\item $\displaystyle\left(\frac{n-(18k+s_0)^2}{p}\right)\ne\left(\frac{5}{p}\right), 0$;
\item $\displaystyle 18k+s_0\nequiv 0\pmod p$.
\end{enumerate}
Since  $ \frac {n-(18k+s_0)^2}9$ is a sum of three squares, there are integers  $a,b,c$, and  $m$ such that $a+b+c=m$ and
$$
n-(18k+s_0)^2=(3a)^2+(3b)^2+(3c)^2=(2m-3a)^2+(2m-3b)^2+(2m-3c)^2.
$$
 Note that at least one of $a,b$, and $c$ is not divisible by $p$. If $abc$ is not divisible by $p$, then $n \stackrel{p}{\lra} I_4$. Assume that exactly two of $a,b$, and $c$ are divisible by $p$. Without loss of generality, assume that both $a$ and $b$ are divisible by $p$. Then neither $m$ nor $(2m-3a)(2m-3b)(2m-3c)$ is divisible by $p$, which implies that $n \stackrel{p}{\lra} I_4$. Assume that exactly one of $a,b$, and $c$ is divisible by $p$. Without loss of generality, we assume that $a$ is divisible by $p$ and $bc$ is not divisible by $p$.  
By changing a sign of $b$, if necessary, we may assume that $a+b+c=m$ is not divisible by $p$. Then clearly, $2m-3a$ is not divisible by $p$. If $2m-3b$ is divisible by $p$, then $n-(18k+s_0)^2\equiv 5m^2\pmod p$. This is a contradiction to (4). 
 Hence $2m-3b$ is not divisible by $p$. Similarly,  we may also show that $2m-3c$ is not divisible by $p$. Therefore $n \stackrel{p}{\lra} I_4$. This completes the proof.
\end{proof}

By combining Propositions \ref{divisible1} and \ref{divisible2}, Theorems \ref{0or2} and \ref{1}, we have the following:

\begin{thm}  Let $p$ be a prime greater than or equal to $5$. Any positive integer $n$ is a sum of  at most $4$ squares not divisible by $p$, except the case when $p=5$ and $n=79$. In the exceptional case, $79$ is a sum of $5$ squares not divisible by $5$. 
\end{thm}

\end{document}